  \documentclass[11pt,letterpaper,twoside]{article}

  \usepackage[margin=3cm]{geometry}

  \usepackage{amsmath,amssymb,amsthm,amsfonts,bbding,stmaryrd,dsfont,wasysym,pifont,xspace,xcolor}
  \usepackage{parskip,fancyhdr,url}

  \usepackage{fourier}

  \usepackage{epstopdf,yfonts}
  \usepackage{graphicx,enumerate,epsfig} 
  \usepackage{mathrsfs}

  \usepackage[
     breaklinks,colorlinks,
     citecolor=black,linkcolor=black,urlcolor=black
  ]{hyperref}

  \usepackage{tikz}
  \usetikzlibrary{arrows,calc,decorations.pathmorphing,backgrounds,positioning,fit}

  \usepackage[all]{xy}

  \usepackage{lipsum}

  \begingroup
    \makeatletter
    \@for\theoremstyle:=definition,remark,plain\do{%
      \expandafter\g@addto@macro\csname th@\theoremstyle\endcsname{%
        \addtolength\thm@preskip\parskip
        }%
      }
  \endgroup

  \newcommand\address[1]{}
  \newcommand\email[1]{}
  \newcommand\dedicatory[1]{}

  \theoremstyle{plain}
  \newtheorem{theorem}{Theorem}[section]
  \newtheorem{proposition}[theorem]{Proposition}
  \newtheorem{corollary}[theorem]{Corollary}
  \newtheorem{lemma}[theorem]{Lemma}

  \theoremstyle{definition}

  \newtheorem*{claim*}{Claim}

  \newtheorem*{question*}{Question}
  \newtheorem*{answer*}{Answer}
  \newtheorem*{application*}{Application}

  \newcommand{\secref}[1]{Section~\ref{Sec:#1}}
  \newcommand{\thmref}[1]{Theorem~\ref{Thm:#1}}
  \newcommand{\corref}[1]{Corollary~\ref{Cor:#1}}
  \newcommand{\lemref}[1]{Lemma~\ref{Lem:#1}}
  \newcommand{\propref}[1]{Proposition~\ref{Prop:#1}}

  \renewcommand{\eqref}[1]{Equation~(\ref{Eq:#1})}

  \newcommand{\Z}{\ensuremath{\mathbf{Z}}\xspace}
  \newcommand{\R}{\ensuremath{\mathbf{R}}\xspace}
  \newcommand{\Q}{\ensuremath{\mathbf{Q}}\xspace}
  \newcommand{\F}{\ensuremath{\mathbf{F}}\xspace}

  \DeclareMathOperator{\scl}{scl}

  \newcommand{\set}[1]{\ensuremath{\left\{ {#1} \right\}}\xspace} 
  \newcommand{\abs}[1]{\ensuremath{\left\lvert {#1} \right\rvert}\xspace} 
  \newcommand{\gen}[1]{\ensuremath{\left\langle {#1}
      \right\rangle}\xspace} 

  \newcommand{\st}{\ensuremath{\,\, \colon \,\,}\xspace} 
  \newcommand{\from}{\ensuremath{\colon \thinspace}\xspace} 

  \newcommand{\B}{\ensuremath{\overline{B}_1}\xspace} 

  \newcommand{\basis}{\ensuremath{\mathcal{B}}\xspace}
  \newcommand{\class}{\ensuremath{\mathcal{C}}\xspace}

  \newcommand{\param}{{\mathchoice{\mkern1mu\mbox{\raise2.2pt\hbox{$
  \centerdot$}}
  \mkern1mu}{\mkern1mu\mbox{\raise2.2pt\hbox{$\centerdot$}}\mkern1mu}{
  \mkern1.5mu\centerdot\mkern1.5mu}{\mkern1.5mu\centerdot\mkern1.5mu}}}

  \fancypagestyle{plain}

  \AtEndDocument{\bigskip{\footnotesize%
  Mathematics Department, University of Oklahoma, Norman, OK 73019, USA
  \par
  \texttt{jing@math.ou.edu} 
  }}

\begin{document}


  \title    {Effective quasimorphisms on free chains}
  \author   {Jing Tao\footnote{\small Partially supported by NSF DMS-1311834.}} 
  \date{}

  \maketitle
  \thispagestyle{empty}
  
  \begin{abstract} 

    \noindent We find homogeneous counting quasimorphisms that are
    effective at seeing chains in a free group \F. As corollary, we derive
    that if a group $G$ has an index--$d$ free subgroup, then every element
    $g \in G$ either has stable commutator length at least $1/8d$ or some
    power of $g$ is conjugate to its inverse. We also show that for a
    finitely-generated free group \F, there is a countable basis for the
    real vector space of homogeneous quasimorphisms on \F.
    

  \end{abstract}
  

\section{Introduction}
  
  Counting quasimorphisms on free groups were first constructed by Brooks
  in \cite{Bro81}. A useful variation of the Brooks' construction was
  introduced by Epstein and Fujiwara in \cite{EF97}. Their variation counts
  the number of disjoint copies of a reduced word in another word (see
  \secref{Chains}). This induces a large family of homogeneous counting
  quasimorphisms on a free group with defect at most 4 \cite{Cal09a}.
  Further, like the Brooks' quasimorphisms, this family \emph{sees} every
  element of \F; that is, for any nontrivial element $ g \in \F$, there is
  a word $w$ whose associated homogeneous counting quasimorphism $\phi_w$
  has $\phi_w(g) \ge 1$. Via Bavard's duality, this yields a lower bound of
  $1/8$ for the stable commutator length (scl) of any nontrivial element in
  a free group.
  
  In this paper, we are interested in quasimorphisms that see chains of a
  free group \F. Let $C_1(\F)$ be the vector space of real $1$--chains of
  \F and let $\overline{C}_1(\F)$ be the quotient of $C_1(\F)$ by the
  subspace spanned by elements of the form $g^n - ng$, and $g-hgh^{-1}$,
  where $g, h \in G$ and $n \in \Z$. Using a lexicographical order on \F,
  we define the notion of \emph{effective words} in \F. Each effective word
  is indivisible (not power of another word) and is a unique representative
  of its own conjugacy class. We show that the set of counting
  quasimorphisms associated to effective words enjoy nice properties. In
  particular, for every nontrivial chain of the form $\sum_i g_i \in
  \overline{C}_1(\F)$, there is an effective word $w$ whose associated
  homogeneous counting quasimorphism $\phi_w$ sees $\sum_i g_i$ (see
  \propref{Chain}). Using Generalized Bavard's Duality \cite{Cal09a} (see
  \thmref{Bavard}), we obtain 
  
  \begin{theorem} \label{Thm:Chains}

    Let \F be any free group and let $g_1,\ldots,g_n \in \F$ be a finite
    sequence such that $\sum_i g_i$ is a $1$--boundary. Then, either
    $\sum_i g_i$ is the zero chain in $\overline{C}_1(\F)$, in which case
    $\scl \left( \sum_i g_i \right) = 0$, or $\scl \left( \sum_i g_i
    \right) \ge 1/8$. 

  \end{theorem}
  
  \thmref{Chains} has applications to the stable commutator lengths in
  virtually free groups. Recall a group is \emph{virtually free} if it
  contains a free subgroup of finite index. Using the finite-index formula
  of \cite{Cal09a} (see \thmref{Formula}), we derive the following
  corollary of \thmref{Chains}. 
  
  \begin{corollary}\label{Cor:VirtuallyFree}

    Let $G$ be a virtually free group with a free subgroup of index $d$.
    Then $\scl(g) = 0$ if and only if there exists $n \in \Z$ such that
    $g^n$ is conjugate to $g^{-n}$. Otherwise, $\scl(g) \ge 1/8d$. 

  \end{corollary}
  
  Since virtually free groups are $\delta$--hyperbolic, it already follows
  from \cite{EF97} (see also \cite{CF10}) that there is always a positive
  lower bound for the scl of every element which does not have a power that
  is conjugate to its own inverse. Their lower bound depends only on
  $\delta$ but it is not explicit. The advantage of \corref{VirtuallyFree}
  is the explicit lower bound relating the index of the free subgroup. 
  
  Finally, we construct a basis for the real vector space of homogeneous
  quasimorphisms on \F using effective words. More precisely, 

  \begin{theorem} \label{Thm:Basis}
    
    For every free group \F, there is a collection of effective words $B
    \subset \F$ such that the associated set $\set{\phi_w}_{w \in B}$ of
    homogeneous counting quasimorphisms is linearly independent. Further,
    for every homogeneous quasimorphism $\phi$ on \F, there are real
    numbers $\set{r_w}_{w \in B} \subset \R$ such that $\phi = \sum_{w \in
    B} r_w \phi_w$. 

  \end{theorem}
  
  We remark that the set $B$ has the property that distinct elements
  represent distinct conjugacy classes, and for every indivisible element
  $g \in B$, either the conjugacy class of $g$ or the conjugacy class of
  its inverse is represented in $B$. As corollary of \thmref{Basis}, we
  obtain that for any countable free group, there is a countable basis for
  the real vector space of homogeneous quasimorphisms. However, it is
  essential that infinitely many of the corresponding coefficients $r_w$
  are allowed to be nonzero, as there are homogeneous quasimorphisms that
  are not a finite linear combination of counting quasimorphism. See the
  end of the paper for an example or Examples 2.63 and 2.64 of
  \cite{Cal09a}.

  Finally, we note that there is a sharp lower bound of $1/2$ for the scl
  of nontrivial elements in free groups. This was first shown in
  \cite{DH91}; (see also \cite{LW14} for another proof). The proof for this
  lower bound does not use quasimorphisms. Thus, it is natural to ask if
  the lower bound of \thmref{Chains} can be improved using a different
  argument.
   
  \subsection*{Acknowledgments}
  
  This paper grew out of a joint project with Talia Fern\'os and Max
  Forester. I am grateful to Max for many stimulating conversations related
  to this paper. I also thank Mladen Bestvina for helpful comments and for
  suggesting the example of the last section. Finally, I thank the referee
  for thorough and helpful comments. 
  
\section{Preliminaries}
  
  We refer to \cite{Cal09a} and references within for background on
  quasimorphisms and stable commutator lengths. 

  Let $G$ be any group. A real-valued function $\Phi \from G \to \R$ is a
  \emph{quasimorphism} on $G$ if there exists $D \ge 0$ such that 
  \[ |\Phi(gh) - \Phi(g) - \Phi(h)| \le D, \quad \forall g,h \in G.\]
  The infimum over all such constants $D$ is the \emph{defect} $D_\Phi$ of
  $\Phi$. It is immediate that $\Phi$ is a homomorphisms if and only if
  $D_\Phi$ is $0$. A quasimorphism is \emph{homogeneous} if it is a
  homomorphism on cyclic subgroups of $G$. Homogeneous quasimorphisms are
  (homogeneous) \emph{class functions}, i.e.\ they are constant on
  conjugacy classes of $G$. Given a quasimorphism $\Phi$, the
  \emph{homogenization} of $\Phi$ is defined by \[ \phi(g) = \lim_{n \to
  \infty} \frac{\Phi(g^n)}{n}.\] This is a homogeneous quasimorphism with
  defect $D_{\phi} \le 2D_\Phi$ \cite[Lemma 2.58]{Cal09a}. 
  
  Let $C_n(G)$ be the real vector space freely generated by $n$--tuples of
  elements in $G$ with  boundary map $\partial \from C_n(G) \to C_{n-1}(G)$
  defined on each $n$--tuple by \[ \partial (g_1,\ldots,g_n) = \sum_{i=1}^n
  (-1)^{i-1} (g_1,\ldots,g_{i-1},g_{i+1},\ldots,g_n).\] A quasimorphism on
  $G$ extends linearly to $C_1(G)$ and a homogeneous quasimorphism descends
  to the quotient $\overline{C}_1(G)$ of $C_1(G)$ by the subspace spanned
  by elements of the form $g^n - ng$, and $g-hgh^{-1}$, where $g, h \in G$
  and $n \in \Z$. We denote by $B_1(G)=\partial C_2(G)$ the subspace of
  real $1$--boundaries of $G$ and let $\B(G)$ be the image of $B_1(G)$ in
  $\overline{C}_1(G)$. For simplicity, we will take the Generalized
  Bavard's Duality as the definition of the \emph{stable commutator length}
  (scl) of an element in $\B(G)$. The definition of the scl of a chain
  (Definition 2.71) and a proof of the duality theorem (Theorem 2.79) can
  be found in \cite{Cal09a}. 

  \begin{theorem}[Generalized Bavard's Duality] \label{Thm:Bavard}
    
    Let $G$ be any group. Then for any $\sum_i^n r_i g_i \in \B(G)$, 
    \[ \scl \left( \sum_i^n r_i g_i \right) = \sup_\phi \frac{ \sum_i^n
    r_i \phi(g_i)} {2 D_\phi}, \] where $\phi$ ranges over all homogeneous
    quasimorphisms on $G$.
  
  \end{theorem}
  
  In the sum above, if $n=1$ and $r_1=1$, then we obtain the usual Bavard's
  duality for $\scl(g)$ of an element $g \in [G,G]$.

  We now describe the finite-index formula stated in \cite[Proposition
  2.80]{Cal09a}. Suppose $H$ is a finite-index subgroup of $G$. Let $X$ be
  a space with fundamental group $G$, and let $\widehat{X}$ be the covering
  space associated to the subgroup $H$. Given a loop $\gamma$ representing
  the conjugacy class of $g\in G$, the preimage in $\widehat{X}$ can be
  expressed as a finite union of loops $\widehat{\gamma}_1, \dots,
  \widehat{\gamma}_k$ each covering $\gamma$ with some finite degree.
  Suppose $\widehat{\gamma_i}$ represents the conjugacy class of $h_i \in
  H$. Then the following formula holds. 

  \begin{theorem}[Finite index formula] \label{Thm:Formula}
    \[
      \scl_G(g) \ = \ \frac{1}{[G:H]} \scl_H \left(\sum_{i=1}^k h_i\right).
    \]
  \end{theorem}

\section{Effective words}
  
  \label{Sec:Chains}

  Let \F be a free group, possibly with infinite rank. Fix a well-ordered,
  symmetric, free generating set for \F and equip \F with the
  lexicographical order. Let $|w|$ be the word length of $w \in \F$. By
  \emph{conjugacy length} $|w|_C$ of $w$ we will mean the word length of a
  cyclically reduced representative of the conjugacy class of $w$.  For
  each cyclically reduced word $w \in \F$, let $[w]$ be the set of
  conjugates of $w$ obtained by cyclically permuting the letters of $w$.
  Note that $[w]$ is a finite set. We will call a word $w \in \F$
  \emph{effective} if $w$ is indivisible and cyclically reduced, and it has
  minimal lexicographical order among all elements in $[w]$. For any
  indivisible element $g \in \F$, we will say $w$ is an \emph{effective
  representative} of $g$ if $w$ is conjugate to $g$ and $w$ is effective. 
  
  Let $w$ be a reduced word. For any $g \in \F$, let $C_w(g)$ be the
  maximum number of disjoint copies of $w$ in the reduced representative of
  $g$. Define
  \[ \Phi_w(g) = C_w(g) - C_{w^{-1}}(g).\]
  By \cite[Proposition 2.30]{Cal09a}, $\Phi_w$ is a quasimorphism with
  defect at most $2$, thus its homogenization $\phi_w$ has defect at most
  $4$.
  
  \begin{lemma} \label{Lem:Effective}
    
    If $w \in \F$ is effective, then $\phi_w(g) = 0$ for any $g \in \F$
    with $|g|_C < |w|$.
    
  \end{lemma}
  
  \begin{proof}
    
    In this proof, we will denote by $w < w'$ to mean $w$ is smaller than
    $w'$ lexicographically. We may assume $g$ is cyclically reduced, so
    $|g| = |g|_C$. Suppose there exists $n> 0$ such that $C_w(g^n) = 1$.
    Since $|g| < |w|$, we must have $n \ge 2$. By cyclically permuting the
    letters of $g$, we may assume $w = (uv)^{n-1}u$, where $g=uv$. By
    assumption, $g$ is cyclically reduced so $|uv| = |vu|$. We can't have
    $uv = vu$, since this would imply $u$ and $v$ are both powers of the
    same word, which contradicts that $w$ is indivisible. Hence, either $uv
    < vu$ or $vu < uv$. If $uv < vu$, then $uuv < uvu$. But $uvu$ is a
    prefix of $w$ and $uuv$ is a prefix of $uwu^{-1} = u(uv)^{n-1}$, so
    $uwu^{-1} < w$, contradicting the minimality of the order of $w$ in
    $[w]$. On the other hand, if $vu < uv$, then $u^{-1}wu = v(uv)^{n-2}uu
    < w$, since $uv$ is a prefix of $w$ but $vu$ is a prefix of $u^{-1}wu$.
    Again, this contradicts minimality of $w$. We conclude $C_w(g^n) = 0$
    for all $n > 0$. Since this is true for all $g$ with $|g| < |w|$, we
    also have $C_{w^{-1}}(g^n) = C_w( (g^{-1})^n) = 0$, for all $n > 0$.
    This shows $\Phi_w(g^n)=0$ for all $n \in \Z$, whence the result.
    \qedhere
    
  \end{proof}
  
  The following lemma is immediate.

  \begin{lemma} \label{Lem:Equal}
    
    Suppose $w\in \F$ is cyclically reduced and $g\in \F$ has $|w| = |g|_C$.
    Then \begin{itemize}
      \item If $g$ is conjugate to $w$, then $\phi_w(g) = 1$. 
      \item If $g$ is conjugate to $w^{-1}$, then $\phi_w(g) = -1$.
      \item In all other cases, $\phi_w(g) = 0$. 
    \end{itemize}

  \end{lemma}

  The following proposition, together with \thmref{Bavard} and the fact
  that $D_{\phi_w} \le 4$ for any reduced word $w \in \F$, yields
  \thmref{Chains}.
  
  \begin{proposition} \label{Prop:Chain} 

    For any finite sequence $g_1,\ldots,g_n \in \F$, if $\sum_i g_i$ is not
    the zero chain in $\overline{C}_1(\F)$, then there is an effective word
    $w \in \F$ with $\sum_i \phi_w \left( g_i \right) \ge 1$.

  \end{proposition}

  \begin{proof}

    Suppose the chain $c=\sum_i g_i$ is nonzero in $\overline{C}_1(\F)$.
    For each $i$, set $g_i = w_i^{m_i}$, where $w_i$ is indivisible and
    $m_i \ge 1$. For $j \ne k$, if $w_j = tw_kt^{-1}$, then $g_j + g_k \equiv
    g_j + tg_kt^{-1} \equiv (m_j+m_k) w_j$ in $\overline{C}_1(\F)$. Similarly,
    if $w^{-1}_j = t w_k t^{-1}$ and $m_j > m_k$, then $g_j + g_k \equiv (m_j -
    m_k) w_j$. Therefore, by applying these equivalence relations whenever
    possible, we may assume $c \equiv \sum_j m_j w_j$, where $w_j$ is not
    conjugate to $w_k$ or $w^{-1}_k$ whenever $j \ne k$ and that each $m_j >
    0$. Finally, by replacing $g_j$ by a conjugate if necessary, we will
    assume each $w_j$ is effective.  
    
    By reordering the indices if necessary, we may assume $|w_1| \ge |w_2|
    \ge \cdots \ge |w_n|$. Let $1 \le k \le n$ be the largest index such
    that $|w_1| = |w_2| = \cdots =|w_k|$. For each $i$, set $\phi_i =
    \phi_{w_i}$. By definition, if $k < n$, then $|w_k| > |w_{k+1}| \ge
    \cdots \ge |w_n|$. Therefore, by \lemref{Effective}, for all $1 \le i
    \le k$ and all $j > k$, $\phi_i(w_j) = 0$. Also, by \lemref{Equal}, for
    all $1 \le i,j \le k$, $\phi_i(w_j) = 0$ if $i \ne j$, and $\phi_i(w_i)
    = 1$. Hence, for each $1 \le i \le k$, \[ \phi_i(c) = \sum_j m_j
    \phi_i(w_j) = m_i  \ge 1. \] Thus choosing $w=w_i$ for any $1 \le i \le
    k$ gives the statement of the proposition. \qedhere 
  
  \end{proof}

  We remark that the proof also shows that for any nonzero chain $c=\sum_i
  r_i g_i \in \overline{C}_1(\F)$, there is an effective word $w \in \F$
  with $\phi_w(c) > 0$. 
  
  \begin{proof}[Proof of \corref{VirtuallyFree}] 
  
    Let $G$ be a virtually free group with a free subgroup $H$ of index
    $d$. By applying the finite index formula (\thmref{Formula}) and
    \thmref{Chains} to $H$, we obtain that if $\scl_G(g) > 0$, then
    $\scl(g) \ge 1/8d$. It remains to show that $\scl_G(g)=0$ if and only
    if some power of $g$ is conjugate to its inverse. Though this follows
    from known results (\cite{EF97} and \cite{CF10}), we will give an
    independent proof that depends only on our results on chains.
    
    If some power of $g$ is conjugate to its inverse then $\scl_G(g)=0$
    follows from homogeneity. Conversely, suppose $\scl_G(g)=0$. Let $N
    \triangleleft G $ be a normal subgroup of finite index contained in
    $H$. The group $N$ is is free, being a subgroup of a free group, so fix
    a well-ordered, symmetric, free generating set for $N$. Applying the
    finite-index formula to $N$, we obtain that 
    \[ \scl_G(g) \ = \ \frac{1}{[G:N]} \scl_N \left(\sum_{i=1}^k
    h_i\right), \] where each $h_i \in N$ is conjugate (in G) to some power
    of $g$. Since $N$ is normal, each $h_i$ is conjugate to the same power
    of $g$. Therefore, it suffices to find a pair $h_i$ and $h_j$ such that
    $h_i$ is conjugate (in $N$) to $h^{-1}_j$. By \thmref{Chains},
    $\scl_G(g) = 0$ if and only $\sum_i h_i$ is trivial in
    $\overline{C}_1(N)$. For each $i$, there exists an effective word $w_i
    \in N$ and $m_i > 0$ such that $h_i \equiv m_i w_i$ in
    $\overline{C}_1(N)$. By reordering the indices if necessary, we may
    assume $|w_1| \ge |w_i|$ for all $i$. If no $w_j$ is conjugate (in $N$)
    to $w^{-1}_1$, then by \lemref{Effective} and \lemref{Equal},
    $\sum_i \phi_{w_1}(h_i) = m_1 > 0$, contradicting $\scl_G(g) =0$.
    Therefore, $w_1$ must be conjugate to the inverse of $w_j$ for some
    $j$. Since $N$ is normal, the conjugacy lengths $|h_i|_C$ are all the
    same. Since each $w_i$ is cyclically reduced, $|h_i|_C = m_i |w_i|$.
    Therefore, $m_1 |w_1| = m_j |w_j|$ which implies $m_1 = m_j$. This
    shows that $h_1$ is conjugate to the inverse of $h_j$ as desired.
    \qedhere

  \end{proof}
  
  \subsection*{Basis}
  
  Let $\widehat{B}$ be the set of effective words of \F. Since an effective
  representative is unique for each conjugacy class, distinct elements of
  $\widehat{B}$ represent distinct conjugacy classes. We decompose
  $\widehat{B} = B \sqcup \overline{B}$, such that $B$ contains the
  effective representative of $g$ if and only if $\overline{B}$ contains
  the effective representative of $g^{-1}$. Note that $B$ forms a basis for
  $\overline{C}_1(\F)$. We impose a well ordering $<$ on $B$ (not the
  lexicographical order) as follows. For $w, w' \in B$, if $|w| < |w'|$,
  then $w<w'$; if $|w| = |w'|$, then set $w < w'$ if $w$ is smaller
  lexicographically. Let $\basis = \set{ \phi_w \st w \in B}$. We show

  \begin{lemma} \label{Lem:Finite}
    
    Only finitely many elements of \basis do not vanish on any nontrivial
    element $g \in \F$.
    
  \end{lemma}
  
  \begin{proof}
    
    For any nontrivial element $g \in \F$, there is a unique element in $B$
    that has a power conjugate to $g$ or $g^{-1}$. Thus, by homogeneity, it is
    enough to show the statement for $g \in B$. For any reduced word $w$
    with $|w| < |g|$, if $\phi_w(g) \ne 0$, then $w$ must be a subword of
    $g^2$. Since there are only finitely many subwords in $g^2$, $\phi_w(g)
    = 0$ for all but finitely many $w \in B$ with $|w| < |g|$. Now by
    \lemref{Effective} and \lemref{Equal}, with the exception of $w=g$, any
    $w \in B$ with $|w| \ge |g|$ has $\phi_{w} (g) = 0$. The conclusion
    follows. \qedhere

  \end{proof}

  Let \class be the real vector space of all homogeneous class functions on
  \F. \thmref{Basis} is immediate from the following proposition. 
  
  \begin{proposition} \label{Prop:Basis}
    
    Any $\phi \in \class$ can be represented by $\phi = \sum_{w \in B} r_w
    \phi_w$, for some $\set{r_w} \subset \R$. Further, \basis is linearly
    independent.

  \end{proposition}

  \begin{proof}
  
    We start with the first statement. 
    For any $w \in B$, denote by $S_w = \set{ w' \in B \st w' < w }$. We
    now use transfinite induction to define $r_w$ for all $w \in B$. For
    the base case, let $w_0 \in B$ be the minimal element and set $r_{w_0}
    = \phi(w_0)$. Now for any $w \in B$, suppose $r_{w'}$ is defined for all
    $w' \in S_w$. Set
    \[ r_w = \phi(w) - \sum_{w' \in S_w} r_{w'} \phi_{w'}(w). \]
    By \lemref{Finite}, the sum in the definition of $r_w$ is well defined
    since there are only finitely many nonzero terms. This defines $r_w$
    for all $w \in B$. To see $\phi = \sum_w r_w \phi_w$, by homogeneity it
    is enough to show that the two sides agree on every element in $B$. By
    \lemref{Effective} and \lemref{Equal}, for any $w,g \in B$ with $w \ge
    g$, $\phi_w(g) = 1$ if and only if $w=g$. Therefore, 
    \[
    \sum_w r_w \phi_w(g) = \sum_{w \in S_g} r_w \phi_w(g) + r_g = \phi(g).
    \] 
    We now show linearly independence of \basis. Suppose there exists
    $\phi_w \in \basis$ which is in the span of $\basis \setminus
    \set{\phi_w}$. 
    We first show that $r_{w'}=0$ for all $w' \ne w$ with $|w'| \le |w|$.
    Fix $k \le |w|$ suppose by induction $r_{w'} = 0$ for all $|w'| < k$.
    Write 
    \[ 
      \phi_w 
      = \sum_{\substack{|w'| \ge k\\ w' \ne w}} r_{w'} \phi_{w'}.
    \]
    Now let $g \in B \setminus \set{w}$ with $|g| = k \le |w|$. By
    \lemref{Effective} and \lemref{Equal}, for
    any $|w'| \ge |g|$, $\phi_{w'}(g)=1$ if and only if $w'=g$ and is
    $0$ for any other case. Therefore, evaluating on $g$, we obtain
    \begin{align*}
    0 = \phi_w(g) 
    & = r_g \phi_g(g)+ \sum_{\substack{|w'| \ge k\\ w' \ne w, g}} r_{w'}
    \phi_{w'}(g) = r_g. 
    \end{align*} 
    This shows $r_{w'}=0$ for all $|w'| \le |w|$. We now have $\phi_w =
    \sum_{|w'| > |w|} r_{w'} \phi_{w'}$. Now plug in $w$ into the
    two sides, we obtain 
    \[ 1 = \phi_w(w) = \sum_{|w'| > |w|} r_{w'} \phi_{w'}(w) = 0. \] 
    This finishes the proof.
  \end{proof}
  
  \begin{corollary}
    
    Let \F be a countable free group. Then the set of class functions with
    rational coefficients is dense in \class in the operator norm; that is,
    for any $\phi \in \class$ and $\epsilon > 0$, there are rational
    numbers $\set{q_w}_{w\in B} \subset \Q$ such that 
    \[ \frac{\abs{ \phi(g) - \sum_w q_w \phi_w(g)}}{|g|} <
    \epsilon, \quad \forall g \in \F\setminus \set{\mathrm{id}}. \]
    Consequentially, the set of homogeneous quasimorphisms with rational
    coefficients is dense in the vector space of homogeneous
    quasimorphisms (in the operator norm or in the defect norm).

  \end{corollary}
  
  \begin{proof}

    Since \F is countable, so is $B$. Let $ \phi = \sum_w r_w \phi_w \in
    \class$ be arbitrary and let $\epsilon >0$ be given. We can choose
    rational numbers $\set{q_w}_{w \in B} \subset \Q$ such that $\sum_w
    |r_w-q_w| < \epsilon$. For any $g \in \F$, $|\phi_w(g)| \le |g|$.
    Therefore
    \[ \abs{ \sum_w \phi_w(g) - \sum_w q_w \phi_w(g)} \le \sum_w
    \abs{r_w-q_w} \abs{g} < \epsilon \abs{g}. \] For the second statement,
    recall that the defect of each $\phi_w$ is at most $4$.
    Thus, if $\phi=\sum_w r_w \phi_w$ is a quasimorphism, then we also have
    that for $\psi= \phi-\sum_w q_w \phi_w$ and for all $g, h \in \F$, 
    \[ \abs{\psi(gh)-\psi(g)-\psi(h)} \le \sum_w |r_w-q_w|
    \abs{\phi_w(gh)-\phi_w(g)-\phi_w(h)} \le 4\epsilon. \] This shows
    $\sum_w q_w \phi_w$ is bounded in the defect norm from $\phi$, hence it
    is a quasimorphism. \qedhere
    
  \end{proof}

  We observe that since each $\phi_w$ has defect at most $4$, if $\sum_w
  |r_w| < \infty$, then $\phi=\sum_w r_w \phi_w$ is a quasimorphism with
  defect at most $4\sum_w |r_w| < \infty$. However, a quasimorphism may
  have non-summable coefficients as illustrated by the following example. 
  
  Let $\F_2=\gen{a,b}$ be the free group of rank $2$. Order its generators
  by $a < a^{-1} < b < b^{-1}$. Choose the set $B$ to contain
  the words $a$ and $a^n b$ for all $n \ge 0$.
  
  For any $g \in \F_2$, represent $g$ in reduced form by $g=
  a^{p_1}b^{q_1}\cdots a^{p_k}b^{q_k}$, where all the exponents are
  non-zero integers except for possibly $p_1$ or $q_k$. Define $\Phi(g)$ to
  be the number of positive even exponents minus the number of negative
  even exponents in the expression for $g$. The map $\Phi$ is a
  quasimorphism by \cite{Rol09}. Let $\phi$ be its homogenization and set
  $\phi = \sum_{w \in B} r_w \phi_w$. We will show by induction that for
  all $n \ge 2$, $r_{a^nb} = (-1)^n$. This immediately implies that
  $\sum_{w\in B} |r_w|$ is not summable. 
  
  First note that \[ \phi(a) = \phi(b) = \phi(ab) = 0 \quad \Longrightarrow
  \quad r_a = r_b = r_{ab} = 0.\] Also, for any $w \in B$, $\phi_w(a^nb)
  \ne 0$ if and only if $w=a$ or $w=a^mb$ for $0 \le m \le n$. In the
  latter case, $\phi_{a^mb}(a^nb) = 1$. Thus, for all $n \ge 2$, we have 
  \begin{align*}
    r_{a^n b} 
    & = \phi(a^n b) - \sum_{w' < a^n b} r_{w'} \phi_{w'}(a^nb) 
    = \phi(a^n b) - \sum_{m=2}^{n-1} r_{a^mb} 
  \end{align*}
  This establishes the base case $r_{a^2b} = \phi(a^2b) = 1$.  We now
  suppose $n > 2$ and for all $2 \le m < n$, $r_{a^mb} = (-1)^m$.
  Therefore, \[ r_{a^nb} = \phi(a^n b) - \sum_{m=2}^{n-1} (-1)^m.\] If $n$
  is odd, then $\phi(a^n b) = 0$, and $\sum_{m=2}^{n-1} (-1)^m =1$. If $n$
  is even, then $\phi(a^n b) = 1$, and $\sum_{m=2}^{n-1} (-1)^m = 0$. This
  yields $r_{a^nb} = (-1)^n$ for all $n \ge 2$ as desired.

 
  \bibliographystyle{alpha} \bibliography{main}

  \end{document}